\documentclass{amsart}

\usepackage{amssymb}
\usepackage{amsmath}
\usepackage{amsthm}
\usepackage{comment}
\usepackage{graphics}

\def\R{\mathbb{R}}

\def\N{\mathbb{N}}

\def\S{\mathbb{S}}

\def\B{\mathbb{B}}

\def\card{\#}

\def\eps{\varepsilon}

\def\Lr{L_r}
\def\Lh{L_h}
\def\LH{L_H}
\def\LM{L_M}

\newtheorem{theorem}{Theorem}[section]
\newtheorem{theorem*}{Theorem}

\newtheorem{proposition}[theorem]{Proposition}

\newtheorem{conjecture*}{Conjecture}

\theoremstyle{definition}

\newtheorem{example}[theorem]{Example}

\theoremstyle{remark}
\newtheorem*{remark*}{Remark}

\begin{document}

\title
[Horospheric limit sets]
{On horospheric limit sets of 
%% normal subgroups of 
Kleinian groups}

\author{Kurt Falk}
\address{Christian-Albrechts-Universit\"at zu Kiel, Mathematisches Seminar,
Ludewig-Meyn-Str. 4, 24118 Kiel, Germany}
\email{falk@math.uni-kiel.de}
\author{Katsuhiko Matsuzaki}
\address{Department of Mathematics, School of Education,
Waseda University, Nishi-Waseda 1-6-1, Shinjuku, Tokyo 169-8050, Japan}
\email{matsuzak@waseda.jp}

\date{}

\begin{abstract}
In this paper we partially answer a question of P. Tukia about the size of the difference 
between the big horospheric limit set and the horospheric limit set of a Kleinian group.
We mainly investigate the case of normal subgroups of Kleinian groups of
divergence type and 
show that this difference is of zero conformal measure by using another
result obtained here: the Myrberg limit set of a non-elementary Kleinian 
group is contained in the horospheric limit set of any non-trivial normal subgroup.
\end{abstract}

\thanks{The authors were supported by JSPS Grant-in-Aid for Challenging Exploratory Research \#16K13767}


\maketitle

%--------------------------------------------------------------------------
\section{Introduction and statement of results}
%--------------------------------------------------------------------------
% (A shorter introduction with statement of the main theorem is required.
% We may use some terminology without definition or with reference to later sections.
% Or otherwise, we can move some parts of later sections here.)

In \cite{sul81a} Sullivan showed that the conservative part of the action of a Kleinian
group $G$ on its limit set coincides up to zero sets of the spherical Lebesgue measure
with the horospheric limit set $\Lh(G)$ of $G$, i.e. the set of all limit points at which every horoball
contains infinitely many orbit points of the group. Later, Tukia \cite{tuk97} generalised this result 
by showing that the same conservative part of the group action on the limit set
coincides with the so-called big horospheric limit set $\LH(G)$ up to zero sets of any conformal
measure of dimension $\delta(G)$ for $G$. Here, $\delta(G)$ denotes as usual the
critical exponent of $G$ and $\LH(G)$ consists of all limit points of $G$ at which there 
exists a horoball containing infinitely many orbit points of $G$. 
%{\sc What has Katsu done in this direction in his follow up paper to Pekka's work?}
For a generalisation of these observations to boundary actions of discrete groups
of isometries of Gromov hyperbolic spaces we refer the reader to Kaimanovich's
work \cite{kai10}.

All this considered, Tukia \cite{tuk97} asked the very natural question of how big
the difference $\LH(G)\setminus\Lh(G)$ might be, also in light of the close
relationship between so-called Garnett points \cite{sul81a} and this difference of sets 
(see \cite{tuk97} for more details).

%SINCE BIG HORO = $L^(1)$ ONE STRATEGY IS TO FISH IN THE STRATIFICATION 
%OF SECTION... NOT EFFECTIVE SINCE...
%In Section~\ref{stratification-section} we describe 
One possible first attempt at answering this question could make use of 
a stratification of the limit set of a Kleinian
group between the radial and the horospheric limit sets in terms of linear escape rates to
infinity within the convex core of the corresponding hyperbolic manifold. These ideas have been 
used by different authors in various slightly different guises and we refer to 
Section~\ref{stratification-section} for details. 
Since both the radial limit set and the big horospheric limit set appear as elements
of this stratification, it would seem that it can be used to measure the difference between the big horospheric
and the horospheric limit set. However, Proposition~\ref{difference-is-small}
goes to show that we cannot detect this difference just by changing linear escape rates.

%IN THIS PAPER WE FOLLOW A SOMEWHAT SURPRISING ROUTE, USING WELL KNOWN 
%NOTION FROM ERGODIC/CONSERVATIVE DYNAMICS
In this paper we follow a different and somewhat surprising idea in order to
answer Tukia's question in the case of normal subgroups of Kleinian groups of
divergence type. We show (see Theorem~\ref{garnett}) that if $N$ is some non-trivial normal
subgroup of the Kleinian group $G$ of divergence type, then $\LH(N)\setminus\Lh(N)$
is a nullset w.r.t. the uniquely determined conformal measure $\mu$ of dimension $\delta(G)$
for $G$. This, of course, is also a conformal measure of dimension $\delta(G)$ for $N$.
We obtain this result as a consequence of another observation (see Theorem~\ref{Myr-in-horo}) 
of independent interest, 
namely, that the Myrberg limit set of a Kleinian group $G$ is always contained in 
the horosperic limit set of any non-trivial normal subgroup $N$ of $G$.  
%REFINEMENT OF KATSU'S ARGUMENT
This in turn is a refinement of the fact proven in \cite{mat02} that in this situation the radial limit 
set of $G$ is contained in the big horospheric limit set of $N$.
%AND OF COURSE MYRBERG OF FULL MEASURE FOR DIVERGENCE TYPE
Of course, one then also needs $G$ to be of divergence type in order to ensure that
the Myrberg limit set is of full $\mu$-measure (see e.g. \cite{tuk94a}, \cite{str97a}
and \cite{falk05} for more details).
%(WHAT IS MYRBERG)
%It also may be worth noting that 
The surprising aspect of this approach is that it is an instance 
where a statement about an essentially
non-conservative phenomenon (the difference between the horospheric and big horospheric limit
set) is proven using a consequence of ergodicity. The Myrberg limit set can be understood as a 
qualitative description of the ergodicity of the geodesic flow in the case of divergence type groups.
We discuss these notions in Section~\ref{preliminaries-section} and the beginning of Section~\ref{Myr-in-horo-section}.

%CONJECTURE 1
Clearly, one wants to measure the difference between the horospheric and big
horospheric limit set also in a more general case than for normal subgroups
of Kleinian groups of divergence type. In view of our previous work \cite{fm15},
we conjecture that the statement of Theorem~\ref{garnett} also holds 
for Kleinian groups for which the convex hull of the limit set admits a uniformly distributed set
whose Poincar{\'e} series diverges at its critical exponent (see the end of 
Section~\ref{Myr-in-horo-section} and Conjecture~\ref{conj1} for more details).

%THIS, THOUGH, LEAVES THE QUESTION OPEN WHY DELTA - BECAUSE THAT 
%APPEARS TO BE THE DIMENSION OF THE MYRBERG LIMIT SET
Having answered Tukia's question for normal subgroups $N$ of groups $G$ of
divergence type by considering the Myrberg limit set 
%$\LM(G)$ 
of $G$, 
%though, 
leaves the question open
why one should measure the size of the difference $\LH(N)\setminus \Lh(N)$ 
by a $\delta(G)$-dimensional conformal measure, as we do,  
and not a $\delta(N)$-dimensional one. 
Recall that $\delta(N)$ may very well be strictly smaller than $\delta(G)$.
The answer, at least in our context, is given in Proposition~\ref{Myr-of-dim-delta} where we
show that the Hausdorff dimension of $\LM(G)$ is equal to $\delta(G)$, provided $G$ 
is of divergence type and the strong sublinear growth limit set $\Lambda_*(G)$ is of
full measure w.r.t. some Patterson measure of $G$ (see Section~\ref{stratification-section} 
and Section~\ref{Myr-of-dim-delta-section} for more details). 
Already Sullivan \cite{sul79a} conjectured 
that $\Lambda_*(G)$ should be of full Patterson measure for all groups $G$ 
of divergence type, but we go one step further and conjecture that 
the Hausdorff dimension of the Myrberg limit set coincides with the critical
exponent for all non-elementary Kleinian groups (Conjecture~\ref{conj2}
in Section~\ref{Myr-of-dim-delta-section}).
One may be able to prove this by refining a well-known argument of 
Bishop and Jones \cite{bj97} showing that the Hausdorff dimension
of the radial limit set is equal to the critical exponent for non-elementary groups.
%CONJECTURE

%BRIEFLY STRUCTURE OF PAPER BY SECTION? Maybe not necessary...

%--------------------------------------------------------------------------
\section{Preliminaries}
\label{preliminaries-section}
%--------------------------------------------------------------------------
\subsection{Limit sets of a Kleinian group}
Let $(\B^{n+1},d)$, $n\geq 1$, be the unit ball model of $(n+1)$-dimensional 
hyperbolic space with the hyperbolic distance $d$. The $n$-dimensional 
unit sphere $\S^n$ is the boundary at infinity of hyperbolic space. 
\emph{Kleinian groups} are discrete subgroups of the group of 
orientation preserving isometries of hyperbolic space. The quotient
$M_G=\B^{n+1}/G$ of $(n+1)$-dimensional hyperbolic space through a torsion free
Kleinian group, that is, a group without elliptic elements, is an 
$(n+1)$-dimensional hyperbolic manifold. 
%For ease of notation, we keep
%the notation $d$ for the projected metric on $M$. 

The \emph{limit set} $L(G)$ of a Kleinian group $G$ is the set of accumulation
points of an arbitrary $G$-orbit, and is a closed subset of $\S^n$. If $L(G)$
consists of more than two points, then it is uncountable and perfect, and
$G$ is called \emph{non-elementary}. The hyperbolic convex hull of the union of all geodesics both
of whose end points are in $L(G)$ is called the \emph{convex hull of} $L(G)$, and
is denoted by $H(L(G))$. The 
quotient $C(M_G) := H(L(G))/G$ is called the \emph{convex core} of 
$M_G$.
Equivalently, the convex core is the smallest convex subset of $M_G$
containing all closed geodesics of $M_G$. A non-elementary 
Kleinian group $G$ is called \emph{convex cocompact} if the convex core
$C(M_G)$ is compact, and \emph{geometrically
finite} if some $\varepsilon$-neighbourhood 
of $C(M_G)$ has finite hyperbolic volume.

A point $\xi \in L(G)$ is a \emph{radial limit point} of $G$ if for any $x \in \B^{n+1}$
and for any geodesic ray 
%$\beta$ 
towards $\xi$ there is a constant $c \geq 0$ such that
infinitely many points of the orbit $Gx$ are within distance $c$ of 
%$\beta$.
the given geodesic ray.
The set of all radial limit points of $G$ is called the
\emph{radial limit set} and is denoted by $L_r(G)$.

A point $\xi \in L(G)$ is a \emph{horospheric limit point} if for any $x \in \B^{n+1}$
every horosphere tangent to $\S^n$ at $\xi$ contains an orbit point $gx$ for some $g \in G$.
The set of all horospheric limit points of $G$ is called the
\emph{horospheric limit set} and is denoted by $L_h(G)$.
A point $\xi \in L(G)$ is an element of the \emph{big horospheric limit set}, denoted
$L_H(G)$, if for $x \in \B^{n+1}$ there is some horosphere tangent to $\S^n$ at $\xi$ that 
contains infinitely many orbit points in $Gx$.
%The set of all big horospheric limit points of $G$ is called the
%\emph{big horospheric limit set} and is denoted by $L_H(G)$.
By definition, we have $L_h(G)\subset L_H(G)$.

\subsection{The critical exponent and invariant conformal measures}
For a Kleinian group $G$ and points $x,z \in \B^{n+1}$, 
the \emph{Poincar{\'e} series} with exponent $s > 0$ is given by
$$
P^s(Gx,z) := \sum_{g \in G} e^{-s \, d(g(x),z)}.
$$
The \emph{critical exponent} $\delta=\delta(G)$ of $G$ is
$$
\delta(G):=\inf \, \{ s>0\,|\,P^s(Gx,z)<\infty\} \\
=\limsup_{R\to\infty} \, \frac{\log \card (B(z,R) \cap Gx)}{R}\, ,
$$
where $B(z,R)$ is the hyperbolic ball of radius $R$ centred at $z$ and
$\card(\cdot)$ denotes the cardinality of a set. By the triangle inequality,
$\delta$ does not depend on the choice of $x,z \in \B^{n+1}$. 
If $G$ is non-elementary, then $0<\delta\leq n$. 
Also, Roblin~\cite{rob02} showed that the above upper limit is in fact a limit.
$G$ is called \emph{of convergence type} if $P^\delta(Gx,z)<\infty$,
and \emph{of divergence type} otherwise.
It is known that a geometrically finite Kleinian group is of divergence type.

A family of positive finite Borel measures $\{\mu_z\}_{z \in \B^{n+1}}$ on $\S^n$ 
is called \emph{$s$-conformal measure} for $s>0$ if $\{\mu_z\}$
are absolutely continuous to each other and, for each $z \in \B^{n+1}$
and for almost every $\xi \in \S^n$,
$$
\frac{d\mu_z}{d\mu_o} (\xi) = |g_z(\xi)|^s,
$$
where $o$ is the origin in $\B^{n+1}$,
$g_z$ is a conformal automorphism of $\mathbb B^{n+1}$ sending $z$ to $o$
and $|\cdot|$ denotes the linear stretching 
factor of a conformal map.
For a Kleinian group $G$, if $\{\mu_z\}$
satisfies $g^*\mu_{g(z)}=\mu_z$ $({\rm a.e})$ for every $z \in \B^{n+1}$ and for every 
$g \in G$, then $\{\mu_z\}$ is called \emph{$G$-invariant}.

The measure $\mu=\mu_o$ can represent the family $\{\mu_z\}$ and 
the $G$-invariance property can be rephrased as follows:
for every $g \in G$ and any measurable 
$A \subset \S^n$ we have
$$
\mu(g(A)) = \int_A |g'(\xi)|^s d\mu(\xi).
$$
%Here, $|g'(\xi)|$ is the unique positive number called the linear stretching 
%factor so that $g'(\xi)/|g'(\xi)|$ is orthogonal. 
If a positive finite Borel measure $\mu$ on $\S^n$
satisfies this condition, we also call $\mu$ itself a \emph{$G$-invariant
conformal measure of dimension $s$}.

We consider a $G$-invariant conformal measure of dimension $\delta=\delta(G)$ supported on
the limit set $L(G)$. The canonical construction of such a measure due to Patterson \cite{pat76}, \cite{pat87}
is as follows (See also \cite{nic89}).
Assume first that $G$ is of divergence type. 
For any $s>\delta$, take a weighted sum of Dirac measures on the orbit $Gx$ for some $x \in \B^{n+1}$:
$$
\mu^s_{(x)}:=\frac{1}{P^s(Gx,o)} \sum_{g \in G} e^{-sd(gx,o)} 1_{gx}.
$$
We can choose some sequence $s_n>\delta$ tending to $\delta$ such that
$\mu^{s_n}_{(x)}$ converges to some measure $\mu$ on $\overline{\B^{n+1}}$ in the weak-$\ast$ sense.
Then we see that $\mu$ is a $G$-invariant conformal measure of dimension $\delta$ supported on $L(G)$,
which is called a \emph{Patterson measure} for $G$.
When $G$ is of convergence type, we need to use a modified Poincar\'e series $\widetilde P^s(Gx,o)$ to make it
divergent at $\delta$ and apply a similar argument.
If $G$ is of divergence type, then a $G$-invariant conformal measure of dimension $\delta$ is 
unique up to multiplication by a positive constant, hence it is {\it the} Patterson measure.

There is another canonical construction of $G$-invariant conformal measures, main\-ly in the case where $G$ is
of convergence type. This was introduced briefly by Sullivan in \cite{sul87}
%given by Anderson, Falk and Tukia 
and developed further in \cite{aft07} (see also \cite{fms10}).
Suppose that the Poincar\'e series for $G$ converges at dimension $s \geq \delta$.
We again consider the weighted sum of Dirac measures $\mu^s_{(x)}$ as above, but here we move
the orbit point $x$ to some point $\xi \in \S^n$ at infinity within a Dirichlet fundamental domain for $G$. 
We can choose a sequence $x_n \in \B^{n+1}$ tending to $\xi$ such that
$\mu^s_{(x_n)}$ converges to some measure $\mu$ on $\overline{\B^{n+1}}$ in the weak-$\ast$ sense.
Then we see that $\mu$ is a $G$-invariant conformal measure of dimension $s$ on $\S^n$,
which is called an \emph{ending measure}.

\subsection{Ergodicity of the geodesic flow}
For a hyperbolic manifold $M_G=\B^{n+1}/G$,
the \emph{unit tangent bundle} $T^1 M_G=\bigsqcup_{p \in M_G}T^1_p M_G$ 
is the union of the unit tangent vectors $v \in T^1_p M_G$ at $p$ taken over all
$p \in M_G$. Each element of $T^1 M_G$ is represented by the pair $(v,p)$.
Let $\tilde g_{\xi,z}(t)$ be a geodesic line of unit speed in $\B^{n+1}$ starting from a given point
$z \in \B^{n+1}$ towards $\xi \in \S^n$ as $t \to \infty$.
The unit tangent bundle $T^1 \B^{n+1}$ of hyperbolic space is also represented by 
$\S^n \times \B^{n+1}=\{(\xi,z)\}$ through the correspondence of 
the unit tangent vector 
$$
\left. \frac{d\tilde g_{\xi,z}}{dt}\right|_{t=0}=\tilde g'_{\xi,z}(0)
$$
to $(\xi,z)$. 
%The projection of $\tilde g_{\xi,z}(t)$ in $M_G$ is denoted by
%$g_{\xi,z}(t)$.
For $(v,p) \in T^1M_G$,
let $g_{v,p}(t)$ denote the geodesic line such that
$g_{v,p}(0)=p$ and $g'_{v,p}(0)=v$.
We can assume that this is the projection of some geodesic line $\tilde g_{\xi,z}(t)$ under 
$\B^{n+1} \to M_G$.
In this case, we also use the notation $g_{\xi,z}$ instead of $g_{v,p}$.
The \emph{geodesic flow}
$\phi_t:T^1M_G \to T^1M_G$
is a map sending $(v,p)$ to $(g'_{v,p}(t),g_{v,p}(t))$
for each $t \in \R$.

Any conformal measure $\mu$ on $\S$ induces a measure $\tilde \mu_*$ on the 
unit tangent bundle $T^1\B^{n+1}=\S^n \times \B^{n+1}$
that is invariant under the geodesic flow (Sullivan \cite{sul79a}, see also \cite{nic89}).
The unit tangent bundle $T^1M_G$ of the hyperbolic manifold $M_G$
is nothing but the quotient of $T^1\B^{n+1}$ by the canonical action of $G$.
If $\mu$ is invariant under $G$, then so is $\tilde \mu_*$ and hence
it descends to a measure $\mu_*$ on $T^1M_G$.

%If $G$ is geometrically finite, then the geodesic flow on $T^1M_G$
%is ergodic with respect to the measure $\nu$. This means that if $E$ is a measurable
%subset of $T^1M_G$ that is invariant under $\phi_t$ for all $t$ then
%$\nu(E)=0$ or $\nu(T^1M_G \setminus E)=0$. From this ergodicity, it follows that
%the time mean coincides with the space mean, that is,
%for every measurable subset $A$ of $T^1M_G$,
%$$
%\lim_{T \to \infty}\frac{1}{T} \int^T_0 1_A(\phi_t(v,p))dt=\frac{\nu(A)}{\nu(T^1M_G)}
%$$
%for almost every $(v,p) \in T^1M_G$ with respect to $\nu$.

We say that the geodesic flow $\phi_t$ is \emph{ergodic} with respect to $\mu_*$ if 
for any measurable
subset $E$ of $T^1M_G$ that is invariant under $\phi_t$ for all $t\in\R$ we have that
either $\mu_*(E)=0$ or $\mu_*(T^1M_G \setminus E)=0$.
Sullivan \cite{sul79a} (and Aaronson and Sullivan \cite{aas84a}) 
generalised the result of Hopf \cite{hop36a}, \cite{hop71a} to show the following
(see also \cite{rob00a}).

\begin{theorem}\label{ergodic}
Let $G$ be a Kleinian group and $\mu$ a $G$-invariant conformal measure of dimension $\delta(G)$.
Then the following conditions are equivalent:
\begin{itemize}
\item[(i)] 
$\mu(\Lr (G))=\mu(\mathbb S^n)$; 
\item[(ii)] 
the geodesic flow $\phi_t$ is ergodic with respect to $\mu_*$; 
\item[(iii)] 
$G$ is of divergence type.
\end{itemize}
%If these conditions hold true, then $\mu$ is unique. 
\end{theorem}

%% Ergodicity of $\phi^t$ w.r.t. $\nu$ (which can be infinite but is always
%% $\sigma$-finite) is best understood using the
%% Hopf Ergodic Theorem: if
%% $f$ and $h$ are integrable functions on
%% $T_1 M_G$ such that $h>0$ and  
%% $\lim_{t \to \infty} \int_{0}^{t} h (\phi^{s}v)\,ds = \infty$ 
%% for $\nu$-almost every $v \in T_1 M_G$, then
%% $$
%% \lim_{t \to \infty} \frac{\int_{0}^{t} f(\phi^s v) \, ds}
%% {\int_{0}^{t} h(\phi^s v) \, ds} =
%% \frac{\int_{T_1 M_G} f \: d\nu}{\int_{T_1 M_G} h \: d\nu}
%% $$
%% for $\nu$-almost all $v \in T_1 M_G$. 
%% (See also \cite{rob00a}.)
%% Now, a central point in this situation is that
%% there exists an integrable function $h_1$ with the
%% properties required above (see \cite{nic89} Theorem~7.2.10 or
%% \cite{sul79a} Proposition~13). 

If $G$ is geometrically finite, then the measure $\mu_*$ corresponding to the
Patterson measure $\mu$ is finite.
If $\mu_*$ is a finite measure, then
the geodesic flow $\phi_t$ is ergodic with respect to $\mu_*$ and hence 
all conditions from Theorem \ref{ergodic} hold (\cite{sul84}).
However, there are also large classes of geometrically infinite groups
for which $\mu_*$ is infinite and these conditions are true 
(\cite{thu79}, \cite{sul81b}, \cite{ree81a} or \cite{aas84a}).
Moreover, there are also examples of geometrically infinite groups
for which $\mu_*$ is a finite measure and the conditions from the theorem
hold true (\cite{pei03}).

%--------------------------------------------------------------------------
\section{%Classification of limit 
Limit sets between radial and horospheric}
\label{stratification-section}
%--------------------------------------------------------------------------

For a Kleinian group $G$, let $\mu$ be a $G$-invariant conformal measure
on $\S^n$ and $X \subset \S^n$ a measurable subset  
that is invariant under $G$.
The action of $G$ is called \emph{conservative} on $X$
with respect to $\mu$ 
if any measurable subset $A \subset X$ with $\mu(A)>0$ satisfies
$\mu(A \cap g(A))>0$ 
for infinitely many $g \in G$.
For the $n$-dimensional spherical measure $\mu$, Sullivan \cite{sul81a} showed that 
$G$ acts conservatively on
the horospheric limit set $L_h(G)$, and that 
the difference between $L_h(G)$ and $L_H(G)$ is actually of null measure.
Later, a characterization of the conservative action for a $G$-invariant conformal measure $\mu$ in general
was obtained by Tukia \cite{tuk97}.
In particular, if a $G$-invariant
conformal measure $\mu$ has no point mass,
then the conservative part $X$, which is the maximal $G$-invariant measurable subset 
of $\S^n$ on which $G$ acts conservatively, coincides 
with the big horospheric limit set $L_H(G)$
up to null sets of $\mu$.

We are interested in the difference $L_H(G) \setminus L_h(G)$, which contains all 
\emph{Garnett points} originally defined 
%depending on the orbit points (see \cite{sul81a} for more details). 
in \cite{sul81a}.
For the spherical Lebesgue measure, 
$L_H(G) \setminus L_h(G)$ is  a null set, but 
Tukia \cite{tuk97} asked how small 
this difference is as measured by a $G$-invariant conformal measure. 
In this section, we will explain
why one should expect that the difference between the big horospheric and the horospheric
limit set is small.

First we introduce a continuous family of limit sets
of a Kleinian group using the approaching order of its orbits.
Fix $c>0$ and $\kappa \in [0,1]$.
For a point $z \in \mathbb{B}^{n+1}$, let $S(z:c,\kappa)$ be
the shadow of a hyperbolic ball
$$
B\left(z,\frac{\kappa}{1+\kappa} d(0,z)+c\right)
$$
w.r.t. the projection from the origin to $\mathbb S^n$ (see \cite{fs04} for more details). 
Essentially the same shadow 
$$
I(z:c,\alpha):=\left\{\xi \in \mathbb S^n \,\left| \;\;   \left|\xi-\frac{z}{|z|}\right |<c(1-|z|)^\alpha 
\right. \right \}
$$
was used in Nicholls \cite{nic89}; $I(z:c,\alpha)$ corresponds to $S(z:c,\kappa)$ via
$\alpha=1/(1+\kappa)$.
%$\alpha=\frac{1}{1+\kappa}$.
For a Kleinian group $G$ acting on $\mathbb{B}^{n+1}$,
consider the orbit $Gz=\{gz\}_{g \in G}$ of
$z \in \mathbb{B}^{n+1}$ and define
$$
\Lr^{(\kappa)}(G):=\bigcup_{c>0} \limsup_{g \in G} S(gz:c,\kappa).
$$
This is the set of points $\xi \in \mathbb S^n$ such that
$\xi$ belongs to infinitely many $S(gz:c,\kappa)$ for some $c>0$.
%which can be defined independently of the choice of $z$.
It is not difficult to see that $\Lr^{(\kappa)}(G)$ is independent of the choice of $z$.

When $\kappa=0$ the set $\Lr^{(\kappa)}(G)$ is nothing more than the radial limit set
$\Lr(G)$, and when $\kappa=1$, $L_r^{(\kappa)}(G)$ coincides with
the big horospheric limit set $L_H(G)$.
By moving $\kappa$ between $0$ and $1$, we  are thus interpolating between 
the radial limit set and the horospheric limit set.

Related limit sets were introduced by Bishop \cite{bis03} and Lundh \cite{lun03}.
We set
$$
\varphi_\xi(t):=d(g_{\xi,z}(t),g_{\xi,z}(0)),
$$
which is the hyperbolic distance in the quotient manifold $M_G$ between $g_{\xi,z}(t)$ 
and the initial point $g_{\xi,z}(0)$. Alternatively, it is defined as the distance of the orbit $Gz$
from $\tilde g_{\xi,z}(t)$ in $\B^{n+1}$.
It is clear that $\varphi_\xi(t) \leq t$. The ratio $\varphi_\xi(t)/t$ measures
how rapidly or slowly the geodesic ray $g_{\xi,z}(t)$ escapes to infinity as $t \to \infty$.
For instance, in Bishop \cite{bis03},
$g_{\xi,z}(t)$ is called a \emph{linearly escaping geodesic} if there exists a positive constant $\kappa>0$
such that $\varphi_\xi(t)/t > \kappa$ for all $t$. 
However, here
we mainly investigate geodesic rays that are escaping to infinity slowly.

For each $\kappa \in [0,1]$ we define the following limit set as the set of
end points of sublinearly escaping geodesic rays:
$$
\Lambda_\kappa(G):=
\{\xi \in \mathbb S^n \mid \liminf_{t \to \infty}\frac{\varphi_\xi(t)}{t} \leq \kappa\}.
$$
The radial limit points correspond to non-escaping geodesic rays and hence
$L_r(G)$ is contained in the \emph{sublinear growth limit set} $\Lambda_0(G)$. 
As an important extremal case, we consider
the \emph{strong sublinear growth limit set}
$$
\Lambda_*(G):=\{\xi \in \mathbb S^n \mid \lim_{t \to \infty}\frac{\varphi_\xi(t)}{t} = 0\},
$$
which is contained in $\Lambda_0(G)$.
%but we do not know the inclusion relation to $L_r(G)$ and neither to $\LM(G)$.
However, while clearly $\LM(G)\subset L_r(G) \subset \Lambda_0(G)$, the 
inclusion relation  of $L_r(G)$ or $\LM(G)$ to $\Lambda_*(G)$ is not a priori clear
(see e.g. Example~\ref{expnew}).

As it turns out, 
%for given $\kappa$, 
the limit sets $\Lr^{(\kappa)}(G)$ 
%can be interpreted geometrically by
and $\Lambda_\kappa(G)$, while not being coincident, are very similar. 
Actually, Lundh \cite[Th.4.1]{lun03} proved that
$$
\Lr^{(\kappa)}(G)=\left\{\xi \in \mathbb S^n \left|\;\liminf_{t \to \infty} 
\left(\frac{1}{1+\kappa}(\varphi_\xi(t)+t)-t\right)<\infty \right.\right\}
$$
for $0 \leq \kappa <1$, and moreover,
$$
\bigcap_{c>0} \limsup_{g \in G} I(gz:c,(1+\kappa)^{-1})
=\left\{\xi \in \mathbb S^n \left|\; \liminf_{t \to \infty} \left(\frac{1}{1+\kappa}(\varphi_\xi(t)+t)-t\right)=-\infty \right.\right\}
$$
for $0 < \kappa \leq 1$. As a consequence, it was shown in \cite[Cor.4.2]{lun03} that
$L_r^{(\kappa)}(G) \subset \Lambda_\kappa(G)$ is always valid, and 
if $\kappa'<\kappa$ then $\Lambda_{\kappa'}(G) \subset L_r^{(\kappa)}(G)$.

When $\kappa=1$, $\Lambda_1(G)$ coincides with
the entire sphere $\mathbb S^n$ since $\varphi_\xi(t)/t \leq 1$ for all $\xi\in\mathbb S^n$ and all $t>0$.
%Namely, $\Lambda_1(G)=\mathbb S^n$. 
Since the horospheric limit set
$L_h(G)$ can be represented by $\bigcap_{c>0} \limsup_{g \in G} I(gz:c,1/2)$, the above result of Lundh implies the following.

\begin{proposition}\label{smallhoro}
$$
L_h(G)=\{\xi \in \mathbb S^n \mid \liminf_{t \to \infty}(\varphi_\xi(t)-t) =-\infty\}.
$$
\end{proposition}

However, 
%the geometric 
a similar dynamical
description of the big horospheric limit set $L_H(G)$ 
%would be more complicated.
is somewhat more involved and we shall deal with this in Section~\ref{bighoro-section}.

It is well known \cite{bj97} that, for any non-elementary Kleinian group $G$,
the Hausdorff dimension of the radial limit set $L_r(G)=L_r^{(0)}(G)$
%is bounded from above by
coincides with the critical exponent $\delta(G)$.
The elementary estimate of the Hausdorff dimension here is the one from above,
and the corresponding argument can be generalised to prove
an upper bound for the Hausdorff dimension ${\rm dim}_H\, L_r^{(\kappa)}(G)$ as follows.
This was proved in \cite[p.575]{fs04}. 
See also \cite[Cor.3]{bis03} and \cite[Prop.4.2]{bmt12} for versions of this statement
formulated for the limit sets $\Lambda_\kappa(G)$. 

\begin{proposition}\label{estimate}
A Kleinian group $G$ satisfies
$$
{\rm dim}_H\, \Lr^{(\kappa)}(G) \leq (1+\kappa)\, \delta(G)
$$
for every $\kappa \in [0,1]$.
\end{proposition}

We have already mentioned that 
%when $\kappa=1$, 
$\Lr^{(1)}(G)=L_H(G)$ and $\Lambda_1(G)=\mathbb S^n$.
However, when $\kappa<1$, the limit sets 
$\Lr^{(\kappa)}(G)$ and $\Lambda_\kappa(G)$ are contained in 
$L_h(G)$ as the following proposition asserts.

\begin{proposition}\label{difference-is-small}
For any Kleinian group $G$ we have 
$$
\bigcup_{0 \leq \kappa<1}\Lr^{(\kappa)} (G)
\subset \Lh(G)
\subset \Lr^{(1)} (G)=L_H(G).
$$
\end{proposition}

\begin{proof}
By the relationship between $L_r^{(\kappa)}(G)$ and $\Lambda_\kappa(G)$,
we see that $\bigcup_{0 \leq \kappa<1}\Lr^{(\kappa)} (G)$ equals
$\bigcup_{0 \leq \kappa<1}\Lambda_\kappa (G)$. For any point $\xi \in \Lambda_\kappa (G)$
with $0 \leq \kappa<1$, its definition gives 
$$
\liminf_{t \to \infty}\frac{\varphi_\xi(t)}{t} \leq \kappa<1.
$$
Since $\varphi_\xi(t)$ is Lipschitz continuous, this implies in particular that
$$
\lim_{t \to \infty}(\varphi_\xi(t)-t)=-\infty.
$$
By Proposition \ref{smallhoro}, we conclude that $\xi \in L_h(G)$.
\end{proof}

%By this proposition, we see 
The statement of Proposition~\ref{difference-is-small} goes to show
that the difference between $L_h(G)$ and $L_H(G)$ is so small
that it cannot be detected by the stratification of the limit set given by the family of sets
$L_r^{(\kappa)}(G)$, $\kappa > 0$.

%--------------------------------------------------------------------------
\section{Dynamical
%Geometric 
characterisation of the big horospheric limit set}
\label{bighoro-section}
%--------------------------------------------------------------------------

In the previous section 
we have seen a 
%geometric 
dynamical
characterisation of  horospheric limit points $\xi \in L_h(G)$
in terms of the distance function $\varphi_\xi(t)$ along the geodesic ray towards $\xi$.
The corresponding result for a big horospheric limit point $\xi \in L_H(G)$ does not have
such a neat form, but we can think of the following claim for $L_H(G)$ as having 
a similar flavour as in the case of $L_h(G)$.

For a geodesic ray $\tilde g_{\xi,z}$ in $\mathbb B^{n+1}$, the {\em Busemann function} $b(x)$ for 
$x \in \mathbb B^{n+1}$ is defined by
$$
b(x):=\lim_{t \to \infty} \{d(x,\tilde g_{\xi,z}(t))-t\}.
$$
Note that the limit always exists since the function taken the limit is bounded from below and decreasing.
A horosphere tangent at $\xi$ is a level set of $b(x)$. For instance, the horosphere passing through $z$
is given by $\{x \in \mathbb B^{n+1} \mid b(x)=0\}$.

\begin{proposition}
\label{bighoro}
For a Kleinian group $G$ and fixed $z \in \mathbb B^{n+1}$, the point
$\xi \in \mathbb S^n$ belongs to $L_H(G)$ if and only if
there exists a sequence $0=t_0<t_1<t_2<\ldots$ converging to infinity and a constant $M$
such that for each $n \in \N$ there is a geodesic segment $\beta_n$ connecting 
$g_{\xi,z}(0)$ and $g_{\xi,z}(t_n)$
of length not greater than $t_n+M$ such that the closed curves in the family
$g_{\xi,z}([0,t_n]) \cup \beta_n$, $n\geq 0$, are mutually freely non-homotopic to each other.
\end{proposition}

\begin{proof}
Assume that $\xi \in L_H(G)$. Then there is a horosphere $H$ and a sequence 
$(\gamma_n)$ in $G$ such that the $\gamma_n(z)$ are inside of the horoball bounded by $H$ 
and converge to $\xi$ as $n \to \infty$. 
We write $H=\{x \in \mathbb B^{n+1} \mid b(x)=M'\}$ for some $M'$ 
by using the Busemann function $b$ for $\tilde g_{\xi,z}$. 
Set $M=M'+\varepsilon$ for some $\varepsilon >0$. By the definition of the Busemann
function, we can find $t_n>0$ for each $n$ such that 
$$
d(\gamma_n(z),\tilde g_{\xi,z}(t_n))-t_n \leq M.
$$
By taking the projection of the geodesic connecting $\gamma_n(z)$ and $\tilde g_{\xi,z}(t_n)$
to $M_G$ as $\beta_n$, we see that this family of closed curves
satisfies the required condition.

Conversely, if we have such a sequence of geodesics $\beta_n$ on $M_G$, 
we lift them to $\mathbb B^{n+1}$
so that, for each $n$, one end point is $\tilde g_{\xi,z}(t_n)$ 
on a fixed geodesic ray $\tilde g_{\xi,z}$ towards $\xi$. 
Then the other end point of the lift of $\beta_n$
is an orbit of $z$ by $G$, which lies inside of the horosphere tangent at $\xi$ given by $b(x)=M$.
Hence we have that $\xi \in L_H(G)$.
\end{proof}

As an application of this claim, we can easily explain the result in
\cite[Th.6]{mat02} corresponding to Theorem \ref{Myr-in-horo}, 
which asserts the inclusion relation $L_r(G) \subset \LH(N)$
for a non-trivial normal subgroup $N$ of a non-elementary Kleinian group $G$.
Indeed, for $\xi \in L_r(G)$, choose a geodesic ray $\hat g_{\xi,z}$ in $M_G$
that returns infinitely often to some bounded neighbourhood of the initial point.
Then its lift $g_{\xi,z}$ to $M_N$ 
%goes 
travels within a bounded distance along 
%images
preimages 
under the covering transformation $M_N \to M_G$
of some fixed closed geodesic. 
If we make a detour at one of these
closed geodesics, we can find a geodesic $\beta_n$ as in Proposition \ref{bighoro}.
We can do this infinitely often in any tail of the geodesic ray, and so
$\xi \in L_H(N)$.

By a similar argument, we have the following consequence from Proposition \ref{bighoro}.

\begin{proposition}\label{boundedgeom}
Let $G$ be a Kleinian group such that the convex core $C(M_G)$ of
$M_G=\mathbb B^{n+1}/G$ has bounded geometry, that is,
there is a constant $M>0$ such that the injectivity radius at every point of $C(M_G)$ is bounded by $M$.
Then every limit point $\xi \in L(G)$ belongs to $L_H(G)$.
\end{proposition}

It was remarked in \cite[Prop.5.1]{fms10} that any $G$-invariant conformal measure $\mu$ does not
have an atom at $\xi \in L_H(G)$ unless $\xi$ is a parabolic fixed point.
In particular, from the above proposition, we see that a non-parabolic ending measure for $G$ has no atom if 
the convex core $C(M_G)$ has bounded geometry.

%--------------------------------------------------------------------------
\section{The Myrberg limit set is contained in the horospheric limit set of a normal subgroup}
\label{Myr-in-horo-section}
%--------------------------------------------------------------------------

In this section we 
%partially 
answer Tukia's question formulated in the introduction and the 
beginning of Section~\ref{stratification-section} 
about measuring the difference
between the big horospheric and the horospheric limit sets for normal subgroups
of Kleinian groups of divergence type, and also give a conjecture generalising the statement
in this case.
The 
%following 
first theorem is the main step towards the answer and is interesting in itself.
Before stating it, we need a few preparations.

A point $\xi \in L(G)$ is a \emph{Myrberg limit point} of $G$ if for any 
distinct limit points $x,y \in L(G)$
and for any geodesic ray $\beta$ towards $\xi$ there is a sequence $\{g_n\} \subset G$ 
such that $g_n(\beta)$ converges to the geodesic line connecting $x$ and $y$.
The set of all Myrberg limit points of $G$ is called the
\emph{Myrberg limit set} and is denoted by $\LM(G)$.
The idea originated in \cite{myr31} and was introduced as a
qualitative version of ergodicity as characterised by Birkhoff's ergodic theorem.
Further developments can be found in \cite{tsu52a}, \cite{nak85a} and \cite{aga83a}, 
and state-of-the-art papers are \cite{tuk94a}, \cite{str97a} and \cite{falk05}.
The main result in the latter three papers is that a Kleinian group being of divergence type 
is equivalent to its Myrberg limit set having full Patterson measure. All geometrically
finite Kleinian groups thus have their Myrberg limit set of full measure.

We can also define the Myrberg limit set by using the geodesic flow on $T^1M_G$. 
Denote the closed subset of
unit tangent vectors that generate geodesic lines staying in the convex core by
$$
%VC=\{(v,p) \in T^1M_G \mid g_{v,p}(t) \in C(M_G) \ (\forall t \in \R)\}.
VC=\{(v,p) \in T^1M_G \mid g_{v,p}(t) \in C(M_G) \text{ for all } t \in \R\}.
$$
Then $\xi \in \LM(G)$ if and only if,
for $p_0 \in C(M_G)$ and $v_\xi$ being the projection of the tangent vector
%to $\B^{n+1}$ 
pointing towards $\xi$ based at some lift of $p_0$,
the forward orbit $\{\phi_t(v_\xi,p_0)\mid t\in\R\}$ of $(v_\xi,p_0)$ 
under the geodesic flow 
%is dense in $VC$.
contains unit tangent vectors that are arbitrarily close to any element of $VC$.

%the image $\phi_t(v_\xi,p_0)$  of the pair $(v_\xi,p_0)$ under the geodesic flow,
%where $p_0 \in C(M_G)$ and $v_\xi$ is the projection of the tangent vector to 
%$\B^{n+1}$ at some lift of $p_0$ which points towards $\xi$, is dense in $VC$.

%the image of the geodesic flow $\phi_t(v_\xi,p_0)$ $(t \in \R)$, where $p_0$ 
%is a point in $C(M_G)$ and $v_\xi$ is the projection of the tangent vector on 
%$\B^{n+1}$ towards $\xi$, are dense in $VC$.
 
%(Define the Myrberg limit set and say something about it
%(`qualitative version of ergodicity as in Birkhoff')).

We consider normal subgroups $N$ of a Kleinian group $G$ and
how properties of limit sets are inherited from $G$ to $N$.
In \cite[Th.6]{mat02}, we have seen the inclusion relation $L_r(G) \subset \LH(N)$.
In the present paper, as 
a refinement of this argument, we prove the following theorem. 

\begin{theorem}
\label{Myr-in-horo}
Let $N$ be a non-trivial normal subgroup of the Kleinian group $G$. Then, 
$$
\LM(G) \subset \Lh(N).
$$
\end{theorem}

We first give a geometric explanation in the manifolds.
Since $N$ is a non-trivial normal subgroup of $G$,
it is non-elementary and hence it contains a loxodromic element $h$.
Let $c$ be the closed geodesic in $M_N$ corresponding to $h$ and 
$\hat c$ the projection of $c$ under the normal covering $M_N \to M_G$, 
which may be a multi-curve.
Take any Myrberg limit point $\xi \in \LM(G)$ and a geodesic ray in ${\mathbb B}^{n+1}$
starting from $z \in {\mathbb B}^{n+1}$ and
towards $\xi \in {\mathbb S}^n$. Then its projection to $M_G$ 
follows $\hat c$ infinitely many times
within an arbitrarily small tubular neighbourhood. 

We consider a lift of this geodesic ray to $M_N$, which is denoted by
$g_{\xi,z}(t)$ $(t \geq 0)$. This goes along some of the images of $c$ under the covering transformations of $M_N \to M_G$ infinitely many times
within arbitrarily small tubular neighbourhoods. This implies that once 
$g_{\xi,z}(t)$ turns around a copy of $c$,
we can find a geodesic between $g_{\xi,z}(0)$ and $g_{\xi,z}(t)$ 
which is shorter than $t$ by some uniform length $\ell>0$.
Namely, $\varphi_\xi(t)=d(g_{\xi,z}(t),g_{\xi,z}(0))$ satisfies $\varphi_\xi(t) \leq t-\ell$
at the time $t$ when we finish one round. 
However, since such detours occur infinitely many times,
we have that $\lim_{t \to \infty} (\varphi_\xi(t)-t)=-\infty$. 
%{\sc (To see that each effect is independent and hence sums up to infinity, 
%we need some rigorous argument here.) }
By Proposition \ref{smallhoro},
this shows that $\xi \in \Lh(N)$.

\begin{proof}[Proof of Theorem~\ref{Myr-in-horo}]
Consider some arbitrary $\xi\in\LM(G)$.
$G$ is assumed to be non-elementary and $N$ to be non-trivial, so $N$ will always contain
hyperbolic elements. Let $n\in N$ be one of these and consider the uniquely determined 
point $o$ on its axis so that the geodesic ray $[o,\xi)$ from $o$ to $\xi$ is orthogonal 
on the axis of $n$. For all $k\in\N$ put $x_k:=n^k(o)$ and $x'_k:=n^{-k}(o)$ and denote 
the geodesic segment connecting $x'_k$ and $x_k$ by $[x'_k,x_k]$.

Given some arbitrary but fixed $\eps>0$, we have 
by the Myrberg property of $\xi$ that there is a sequence $(g_k)_{k\in\N}$ of elements
of $G$ so that $g_k(o)$ tends to $\xi$ in the Euclidean metric and, for all $k\in\N$, 
the geodesic segment $g_k([x'_k,x_k])$ is \emph{$\eps$-close} to $[o,\xi)$, 
meaning that any point on $g_k([x'_k,x_k])$ is within distance $\eps$ from the 
geodesic ray $[o,\xi)$.

\begin{figure}
\includegraphics{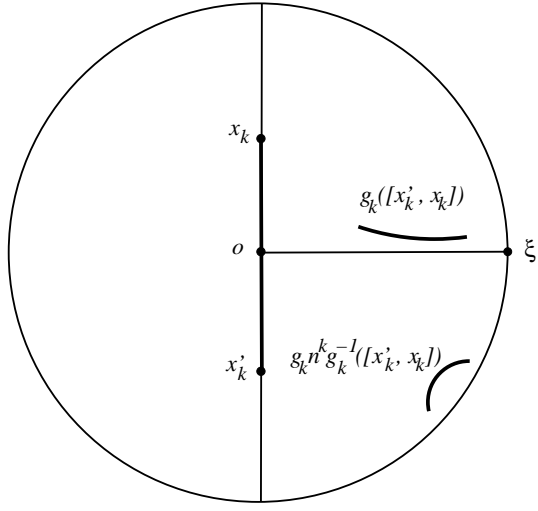}
    \label{Myr-in-horo-fig}
    \caption{The setting of Theorem~\ref{Myr-in-horo}.}
\end{figure}

Since $N$ is normal in $G$, we know that $g_k n^k g_k^{-1}, g_k n^{-k} g_k^{-1}\in N$ for all $k\in \N$. 
A priori, it is not clear how $g_k([x'_k,x_k])$ is `oriented' with respect to $[o,\xi)$, but we 
shall see that this is not important for the argument. For a given $k\in\N$, assuming $g_k(x'_k)$ is 
closer to $o$ than $g_k(x_k)$, we have that
\begin{align*}
d (g_k n^k g_k^{-1}(x_k), g_k(o))
&= 
d(n^k g_k^{-1}(x_k), o) 
= 
d(n^k g_k^{-1}(x_k), n^k(x'_k)) \\
&=\ 
d(g_k^{-1}(x_k), x'_k)  
= 
d (x_k, g_k(x'_k)) \\
&\asymp_+ d(o,g_k(o))
\end{align*}
and that
\begin{align*}
d (g_k n^k g_k^{-1}(x'_k), g_k(x_k))
&= 
d (n^k g_k^{-1}(x'_k), x_k)
=
d(n^k g_k^{-1}(x'_k), n^k(o)) \\
&= 
d(g_k^{-1}(x'_k), o)   
= 
d (x'_k, g_k(o)) \\
&\asymp_+
d(o,g_k(x_k)).
\end{align*}
Here, the additive comparabilities $\asymp_+$,
which mean that the difference between the comparable distances is uniformly bounded
independently of $k$,
are due to the fact that 
$[o,\xi)$ is orthogonal on the axis of $n$ and $g_k([x'_k,x_k])$
is $\eps$-close to $[o,\xi)$. 
If $g_k(x_k)$ is closer to $o$ than $g_k(x'_k)$, then the same argument as above yields that
$$
d (g_k n^{-k} g_k^{-1}(x'_k), g_k(o)) \asymp_+ d(o,g_k(o));\ 
d (g_k n^{-k} g_k^{-1}(x_k), g_k(x'_k)) \asymp_+ d(o,g_k(x'_k)).
$$

From these estimates, elementary hyperbolic geometry shows that
there is some horosphere $H$ tangent at $\xi$ such that either both 
$g_k n^k g_k^{-1}(x_k)$ and 
$g_k n^k g_k^{-1}(x'_k)$ or both $g_k n^{-k} g_k^{-1}(x_k)$ and 
$g_k n^{-k} g_k^{-1}(x'_k)$
lie inside of $H$ for all $k$.
Since
$$
d(g_k n^k g_k^{-1}(x_k),g_k n^k g_k^{-1}(x'_k))=
d(g_k n^{-k} g_k^{-1}(x_k),g_k n^{-k} g_k^{-1}(x'_k))=
d(x_k,x'_k) \to \infty 
$$
as $k \to \infty$, the mid points $g_k n^k g_k^{-1}(o)$ or $g_k n^{-k} g_k^{-1}(o)$
of the geodesic segments enter smaller and smaller horospheres tangent at $\xi$.
This can be easily seen if we use the upper half-space model of the hyperbolic space
with $\xi$ being infinity. Hence we have that $\xi \in L_h(N)$. 
\end{proof}

As a direct corollary to Theorem~\ref{Myr-in-horo}, we obtain that the difference between
the big horospheric and the horospheric limit sets of $N$ is a null set for 
the Patterson measure $\mu$ for $G$ when $G$ is of divergence type. 
%Here, since $G$ is assumed to be of divergence type,
%$\mu^{\delta(G)}$ is the unique conformal measure of
%dimension $\delta(G)$ and total mass $1$ on $L(G)$ (cite Sullivan!!). 
Also, since $N$ is normal in $G$, we have that $L(N)=L(G)$ and that $\mu$ is a 
conformal measure of dimension $\delta(G)$ for $N$ as well. 
Note that $\delta(N)$ may very well 
be strictly smaller than $\delta(G)$, in particular, 
when $G$ is convex cocompact and $G/N$ is non-amenable (\cite{bro85}).
Concerning the investigation of this phenomenon in view of the dimension gap 
between $L(N)$ and $L_r(N)$, see for instance \cite{fm15} or the survey \cite{falk18}.

\begin{theorem}
\label{garnett}
Let $N$ be a non-trivial normal subgroup of the Kleinian group $G$, and assume that $G$ is of 
divergence type. Then,  
$$
\mu^{\delta(G)}(\LH(N)\setminus \Lh(N))=0
$$
for some $N$-invariant conformal measure $\mu^{\delta(G)}$ of dimension $\delta(G)$.
\end{theorem}

The following class of examples illustrates the statements of Theorem~\ref{Myr-in-horo} 
and Theorem~\ref{garnett}  in a non-trivial way.

\begin{example} %(\cite{fs04})
\label{expl1}
Let $G_{0}$ and $G_{1}$ be Schottky groups 
%%%and 
%assume $G_{0}$ is freely generated by hyperbolic
%isometries $g_{1},\ldots, g_{k}$, and that 
%$G_{1}$ is freely generated by more than one hyperbolic isometry.
%Consider 
%%%consider the (open) fundamental domains
%%%$F_{0}$ and $F_{1}$ defined by the generators of $G_{0}$ and $G_{1}$,
%%%respectively, and the corresponding choice of hyperplanes.
%Furthermore, assume 
%%%Assume that these fundamental domains have 
with fundamental domains having disjoint complements in 
hyperbolic space,
%%% i.e. 
%%%$F_{0}^{{\bf c}} \cap F_{1}^{{\bf c}} = \emptyset$, 
and define
$G:= G_0 * G_1$ which is then also a Schottky group.
Put $N:= \ker(\varphi )$, where $\varphi : G \to G_{1}$ is the canonical 
group homomorphism. Thus, $0\to N\to G\to G_1\to 0$ 
is a short exact sequence,
%and 
%$$
%N = \langle hg_ih^{-1}: i=1,...,k, h \in G_{1} \rangle .
%$$ 
%Furthermore, 
$N$ is the normal subgroup of $G$ generated by
$G_{0}$ in $G$, and 
$G / N \cong G_{1}$. Clearly, 
$N$ is infinitely generated and if we assume that $G_1$ is freely generated 
by at least two generators,
and is thus non-amenable, then the already mentioned result of Brooks~\cite{bro85} ensures that
$\delta(N)<\delta(G)$. For more details on this class of examples see also \cite{fs04}.
Theorem~\ref{Myr-in-horo} now applies and we thus have that 
$\LM(G) \subset \Lh(N)$. We also know on the one hand
that in this situation  the Hausdorff dimension of $\LM(G)$ coincides
with $\delta(G)$ since $G$ is cocompact and thus $\LM(G)$ is of full measure w.r.t. 
the Patterson measure of $G$ (\cite{tuk94a}, \cite{str97a} and \cite{falk05}) which is 
known \cite{sul79a} to be proportional to the $\delta(G)$-dimensional Hausdorff measure on 
$L(G)$. On the other hand, we know by \cite{bj97} that $\dim_H(\Lr(N))=\delta(N)$.
We thus know that $\Lr(N)$ has strictly smaller Hausdorff dimension than both
$\LM(G)$ and $\Lh(N)$ which makes the statements of both 
Theorem~\ref{Myr-in-horo} and Theorem~\ref{garnett} meaningful and non-trivial.
%%%%% its horospheric limit set is larger than its radial limit set.
\end{example}

Concerning the difference between the big and the small horospheric limit sets in 
Theorem~\ref{garnett},
we are considering the situation where $N$ is contained as a normal subgroup 
in some Kleinian group $G$, and the limit sets are measured by a conformal measure for $G$. However, it is desirable to describe
the difference between these limit sets only by using the Kleinian group in question itself.
Here is an idea how to do this, which makes use of our previous work \cite{fm15}.

We call a discrete $G$-invariant set $X=\{x_i\}_{i=1}^{\infty}$ in the convex hull 
$H(L(G))$ %%\subset \B^{n+1}$ 
of $L(G)$ \emph{uniformly distributed} if the following two 
conditions are satisfied:
\begin{itemize}
\item [(i)]
There exists a constant $M<\infty$ such that, for every point 
$z \in H(L(G))$, there is some $x_i \in X$ such that $d(x_i,z)\leq M$;
\item [(ii)]
There exists a constant $m>0$ such that, any distinct points 
$x_i$ and $x_j$ in $X$ satisfy $d(x_i,x_j) \geq m$.
\end{itemize}

For a uniformly distributed set $X$, we define the 
\emph{extended Poincar\'e series} with exponent $s>0$ and reference point 
$z \in \B^{n+1}$ by
$$
P^s(X,z):=\sum_{x \in X}e^{-s\,d(x,z)}.
$$
The \emph{critical exponent} for $X$ is
$$
\Delta :=\inf \{s>0 \mid P^s(X,z)<\infty\}.
$$
The Poincar\'e series for $X$ is of \emph{convergence type} if $P^\Delta(X,z)<\infty$, and of
\emph{divergence type} otherwise.
Moreover, we can define the associated Patterson measure $\mu_X$ for $X$
supported on $L(G)$ by a similar construction to the usual case.

As a sufficient condition for the extended Poincar\'e series 
$P^s(X,z)$ to be of divergence type, we have the following.
A uniformly distributed set $X$ is of {\em bounded type}
if there exists a constant $\rho \geq 1$ such that
$$
\frac{\card (X \cap B_R(x))}{\card (X \cap B_R(z))} \leq \rho
$$
for every $x \in X$ and for every $R>0$. 
In this case, the $\Delta$-dimensional Hausdorff measure of $L(G)$ is positive and
$P^\Delta(X,z)=\infty$. For more details see \cite{fm15}.

In view of these similarities to the case where our group in question is a normal subgroup of
some Kleinian group of divergence type, 
we give the following conjecture in analogy to Theorem \ref{garnett}.

\begin{conjecture*}
\label{conj1}
If $G$ is a Kleinian group whose convex hull $H(L(G))$ admits a uniformly distributed set $X$
so that the extended Poincar\'e series $P^s(X,z)$ is of divergence type, then
$$
\mu_X(\LH(G)\setminus\Lh(G))=0
$$
for the associated Patterson measure $\mu_X$.
\end{conjecture*}

If $L(G)=\mathbb S^n$, then we can choose the Lebesgue measure 
on $\mathbb S^n$ as $\mu_X$.
In this case, the original result of Sullivan \cite{sul81a} on Garnett points 
supports the conjecture.

%--------------------------------------------------------------------------
\section{The Hausdorff dimension of the Myrberg limit set}
\label{Myr-of-dim-delta-section}
%--------------------------------------------------------------------------

%Why did we look at conformal measures of dimension $\delta(G)$ in the previous theorem 
%to measure the difference between the big horospheric and horospheric limit sets of $N$? 
%Because this should be the Hausdorff dimension of these sets. 

In this section we justify why in the previous section we considered conformal measures 
of dimension $\delta(G)$ in order to measure the difference between
the big horospheric and horospheric limit sets of $N$. 
Namely, we will show under a certain assumption
that the Hausdorff dimension of $\LM(G)$,
and thus, in view of Theorem~\ref{Myr-in-horo}, of both $\Lh(N)$ and $\LH(N)$, 
is equal to $\delta(G)$.

\begin{proposition}
\label{Myr-of-dim-delta}
If $G$ is a Kleinian group of divergence type such that the strong sublinear
growth limit set $\Lambda_*(G)$
has full measure for the Patterson measure $\mu$ of $G$, then
$$
\dim_H(\LM(G))=\delta(G).
$$
All assumptions on $G$ follow from the condition 
$\mu_*(T_1 M_G)<\infty$. 
\end{proposition}

The method of proof is a generalisation of the argument for the radial limit set
given in Sullivan \cite{sul79a}. (See also Nicholls \cite[Th.9.3.5]{nic89}.) Note that 
Sullivan already conjectures in the original paper that for any divergence type group $G$
the strong sublinear
growth limit set $\Lambda_*(G)$ is of full measure for the Patterson measure of $G$. 
That is, Proposition \ref{Myr-of-dim-delta} should be valid without 
the extra assumption on $\Lambda_*(G)$.
One can justifiably ask why this is not clear for divergence type groups in general.
The reason is that for geometrically infinite groups $G$, the radial limit set
$L_r(G)$ is not necessarily
contained in $\Lambda_*(G)$ as the following example shows.
However, we expect that $\LM(G) \subset \Lambda_*(G)$ should be true. This can still be
regarded as a generalisation of Sullivan's conjecture.

\begin{example}
\label{expnew}
%(See also Figure~\ref{expl}.)
Let $T$ be a once-punctured torus, and
let $a$ and $b$ be simple closed geodesics on $T$
whose intersection number is 1.
We cut open $T$ along $a$ to obtain a bordered surface $P$ with 
%1
one puncture and
%2
two geodesic boundary components.

We prepare infinitely many copies of $P$ and paste them one after another along the geodesic boundary components without a twist.
The resulting surface is a cyclic cover of $T$, which is
denoted by $R$. Let $\langle h \rangle$ be the covering transformation group.
The lift of $b$ to $R$, which is a geodesic line invariant under $\langle h \rangle$, is denoted by
$\tilde b$. We also take a simple closed geodesic $a_0$ that is a component of the lift of $a$ and
set the intersection of $\tilde b$ and $a_0$ as a base point $o$. Set $a_n=h^n(a_0)$ and $o_n=h^n(o)$ for
every integer $n$.

We consider the following infinite curve starting at $o$: 
$$
\beta_0=\prod_{k=0}^\infty (\tilde b[o,o_{(-2)^k}] \cdot a_{(-2)^k} \cdot \tilde b[o_{(-2)^k},o]).
$$
Here,
$\tilde b[x,y]$ denotes the segment in $\tilde b$ from $x$ to $y$. 
Then, we take the geodesic ray $\beta:[0,\infty) \to R$ starting from $o$ and going to infinity navigated homotopically by 
$\beta_0$.
Since $\beta$ returns infinitely many times to some neighbourhood of $o$, the end point of $\beta$ gives
a radial limit point $\xi \in L_r(G)$ of a Fuchsian group $G$ uniformising $R$.

\begin{figure}
\includegraphics{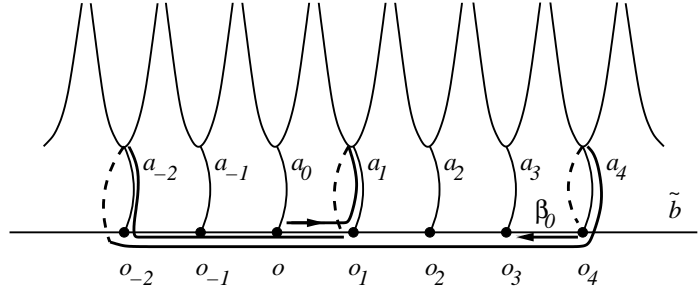}
    \label{expl}
    \caption{The setting of Example~\ref{expnew}.}
\end{figure}

On the other hand, $\xi$ does not belong to the strong sublinear growth limit set $\Lambda_*(G)$ of $G$.
To see this, for every $n \geq 1$,
let $t_n > 0$ be the arc length parameter of 
the geodesic ray $\beta=g_{\xi,o}$ such that $\beta(t)$ crosses over $a_{(-2)^n-(-1)^n}$ for the first time. 
We denote the hyperbolic length by $\ell(\cdot)$ and the hyperbolic distance by $d(\cdot,\cdot)$.
Then, we have that
\begin{align*} 
t_n &\leq \sum_{k=0}^{n-1} \{2\ell(\widetilde b[o,o_{(-2)^k}])+\ell(a)\}+\ell(\widetilde b[o,o_{(-2)^n}])+\ell(a)\\
&\leq 3 \cdot 2^n \ell(b)+(n+1)\ell(a)\\
&< 3 \cdot 2^n \{\ell(b)+\ell(a)\}.
\end{align*}
However, since
$d(o, \beta(t_n)) \geq (2^n-1) d(a_0,a_1)$, we obtain that
$$
\frac{d(o, \beta(t_n))}{t_n} > \frac{1}{4}\frac{d(a_0,a_1)}{\ell(b)+\ell(a)}
$$
for every $n \geq 1$. This implies that
$$
\limsup_{t \to \infty} \frac{\varphi_\xi(t)}{t} >0,
$$
and hence $\xi \notin \Lambda_*(G)$.
\end{example}

\begin{proof}[Proof of Proposition~\ref{Myr-of-dim-delta}]
It is proved in \cite{tuk94a}, \cite{str97a}
and \cite{falk05} that
for a Kleinian group $G$ of divergence type,
the Myrberg limit set $\LM(G)$ has full measure w.r.t. the Patterson measure $\mu$ of $G$.
Moreover, by assumption, the sublinear
growth limit set $\Lambda_*(G)$
has full $\mu$-measure. 

Following the argument from \cite[Lemma 9.3.4]{nic89},
we can find a compact subset $K$ of $\LM(G) \cap \Lambda_*(G)$ with
$\mu(K)>0$ satisfying the following property: for any $\varepsilon>0$ there exists $r_0>0$ 
such that if $\xi \in K$ and $r<r_0$ then 
$$
\mu(B(\xi,r) \cap K)/r^{\delta(G)-\varepsilon}<A,
$$
where $A$ is some absolute constant.
From this property, it follows that $K$ has positive $(\delta(G)-\varepsilon)$-dimensional
Hausdorff measure for any $\varepsilon>0$ (see \cite[Th.9.3.5]{nic89}). 
Hence $\dim_H \LM(G) \geq \delta(G)$. 
The converse inequality is clear from $\dim_H L_r(G) = \delta(G)$ and $\LM(G) \subset L_r(G)$, 
and thus the first statement follows.

To verify the latter statement, it suffices to remark that
$\mu_*(T_1 M_G)<\infty$ implies 
that $G$ is of divergence type and that $\Lambda_*(G)$ has full $\mu$-measure
(see \cite[Cor.19]{sul79a}, \cite[Lemma 9.3.3]{nic89}).
\end{proof}

\begin{example} 
%\label{expl2}
Here it is interesting to mention a class of non-trivial examples 
for which the statement of Proposition~\ref{Myr-of-dim-delta} applies.
In \cite{pei03} Peign{\'e} constructs geometrically infinite Schottky groups $G$
of divergence type which at the same time satisfy $\mu_*(T_1 M_G)<\infty$.
For more details we refer the interested reader to the original article \cite{pei03}.
\end{example}

Following a completely different idea of proof, 
it may be possible to generalise the argument in Bishop and Jones \cite{bj97} showing that
the the Hausdorff dimension of the radial limit set coincides with the critical
exponent, in order to prove the following conjecture.

\begin{conjecture*}
\label{conj2}
For any non-elementary Kleinian group $G$ we have
$$
\dim_H(\LM(G))= \delta(G).
$$
\end{conjecture*}

\end{document}